\documentclass{amsart}
\numberwithin{equation}{section}

\usepackage{bbm}
\usepackage{comment}
\usepackage{verbatim}
\usepackage{mathrsfs}
\usepackage{mathtools}
\usepackage{latexsym}
\usepackage{stmaryrd}
\usepackage{epsfig}
\usepackage{amsmath}
\usepackage[usenames,dvipsnames]{xcolor}
\usepackage{bbm}
\usepackage{graphics}
\usepackage[utf8]{inputenc}
\usepackage{amssymb}
\usepackage{amsthm}
\usepackage[alphabetic]{amsrefs}
\usepackage{amsopn}
\usepackage{amscd}
\usepackage[all,knot]{xy}
\xyoption{all}
\usepackage{rotating}
\usepackage{tikz}
\usetikzlibrary{calc}
\usepgflibrary{shapes.geometric}
\usepgflibrary{shapes.misc}
\usetikzlibrary{positioning}
\usetikzlibrary{decorations}
\usetikzlibrary{decorations.pathreplacing}
\usepackage{hyperref}
\usepackage{tikz}
\usepackage{tikz-cd}
\usepackage{microtype} 
\mathtoolsset{showonlyrefs}

\theoremstyle{plain}
\newtheorem{thm}{Theorem}

\newtheorem{prop}[thm]{Proposition}
\newtheorem{cor}[thm]{Corollary}

\newtheorem*{thm*}{Theorem}
\newtheorem*{lem*}{Lemma}
\newtheorem*{prop*}{Proposition}
\newtheorem*{cor*}{Corollary}

\theoremstyle{definition}

\newtheorem*{defn*}{Definition}

\newtheorem{ex}[thm]{Example}
{}

\newtheorem*{rem*}{Remark}
\newtheorem*{war*}{Warning}

\newtheorem*{hyp_plain*}{Hypotheses}
{}
{}
{}
\newtheorem*{ack}{Acknowledgements}{}
{}

\theoremstyle{remark}
{}
{}
{}

\def\to{\longrightarrow} 

\def\ZZ{\mathbb{Z}}
\DeclareUnicodeCharacter{221E}{$\infty$}

\def\sfD{\mathsf{D}}

\def\sfH{\mathsf{H}}

\DeclareMathOperator{\colim}{colim}

\DeclareMathOperator{\Hom}{Hom}

\DeclareMathOperator{\modu}{\mathsf{mod}}

\DeclareMathOperator{\Tor}{Tor}

\DeclareMathOperator{\coh}{coh}

\DeclareMathOperator{\Spec}{Spec}

\definecolor{internationalkleinblue}{rgb}{0.0, 0.18, 0.65}

\title{Rouquier dimension versus global dimension}

\author{Greg Stevenson}
\address{Greg Stevenson, Aarhus University, Department of Mathematics, Ny Munkegade 118, bldg. 1530
DK-8000 Aarhus C, Denmark
}
\email{greg@math.au.dk}

\begin{document}
\begin{abstract}
\noindent We give an example of a commutative coherent ring of infinite global dimension such that the category of perfect complexes has finite Rouquier dimension.

\end{abstract}
%

\maketitle

\setcounter{tocdepth}{1}





In \cite{neeman2017strong}*{Remark~0.11} Neeman claims that strong generation of the category of perfect complexes over a commutative ring implies finite global dimension (without assuming the ring is noetherian as emphasized in loc.\ cit.\ Remark~0.10). This is a special case of Theorem~0.5 of the same paper which asserts that for $X$ a quasi-compact and separated scheme the category $\sfD^\mathrm{perf}(X)$ is strongly generated if and only if $X$ has a cover by open affine subschemes corresponding to rings of finite global dimension. 

No argument is given for the affine case discussed in the Remark, but there is a reference to Rouquier \cite{rouquier08}*{Proposition~7.25}. However, Rouquier's statement is about the bounded homotopy category of \emph{all} projectives and the bounded derived category of \emph{all} modules. In particular, it does not apply to the question at hand. 

We give an example showing Remark~0.11 and Theorem~0.5 of \cite{neeman2017strong} are false: there is a coherent commutative ring of infinite global dimension such that its category of perfect complexes has Rouquier dimension $0$. On the other hand, we prove that for a coherent ring strong generation of the perfect complexes is equivalent to finite \emph{weak} global dimension.

\section{The examples}

Although our examples will be commutative the general facts we need are valid without this assumption. Let $R$ be a von Neumann regular ring. This is equivalent to the property that every $R$-module is flat. We will make use of the following standard facts.

\begin{prop}
Any von Neumann regular ring $R$ satisfies the following properties:
\begin{itemize}
\item[(i)] $R$ is coherent;
\item[(ii)] every finitely presented $R$-module is projective.
\end{itemize}
\end{prop}
\begin{proof}
See \cite{Prestbook}~{Corollary~2.3.23}.
\end{proof}

Thus it makes sense to talk about $\sfD^\mathrm{b}(\modu R)$ the bounded derived category of finitely presented $R$-modules. This turns out to be simple, at least from the perspective of generation, because of the projectivity of finitely presented modules.

\begin{prop}\label{prop:Rdim}
Every object of $\sfD^\mathrm{b}(\modu R)$ is formal. In particular, $\sfD^\mathrm{b}(\modu R) = \sfD^\mathrm{perf}(R)$ has Rouquier dimension $0$.
\end{prop}
\begin{proof}
Let $X$ be an object of $\sfD^\mathrm{b}(\modu R)$. By (ii) of the previous proposition the cohomology groups of $X$ are projective. In particular, for each $i$ we have an isomorphism $\Hom(\Sigma^{-i}\sfH^i(X), X) \cong \Hom(\sfH^i(X),\sfH^i(X))$ of morphism spaces in $\sfD^\mathrm{b}(\modu R)$ and so we get a canonical map
\[
\bigoplus_{i\in\ZZ}\Sigma^{-i}\sfH^i(X) \to X,
\]
which induces the identity on each cohomology group, i.e.\ is a quasi-isomorphism. Thus $X$ is formal.

The other conclusions follow directly: we have just proved that every bounded complex of finitely presented $R$-modules is quasi-isomorphic to a sum of shifts of finitely generated projectives and hence perfect. This also implies that $R$ generates under shifts and summands and hence the perfect complexes have Rouquier dimension $0$.
\end{proof}

Now let us specialize to some particular choices of $R$. Our source for the first two examples is \cite{OsofskyVNR}, see in particular Section~3 therein. In each case we consider $R = \prod_\lambda k$ for a field $k$ and some choice of cardinal $\lambda$. We note that any ring of this form is coherent and von Neumann regular.

\begin{ex}
Suppose that we take $\lambda = \aleph_\omega$. Then $R$ has infinite global dimension. 
\end{ex}

\begin{ex}
Suppose that we take $\lambda = \aleph_0$. Then the global dimension of $R$ is $k+1$ if $2^{\aleph_0} = \aleph_k$ for $k$ a natural number, and it is infinite if $2^{\aleph_0} \geq \aleph_\omega$. 
\end{ex}

In fact one can produce (commutative) von Neumann regular rings of any global dimension without worrying about one's model of set theory.

\begin{ex}
By \cite{MR0229695} the free Boolean algebra on $\aleph_n$ generators has global dimension $n+1$.
\end{ex}

We conclude that finiteness of the Rouquier dimension of $\sfD^\mathrm{perf}(R)$ cannot determine whether or not $R$ has finite global dimension. In fact, all of these examples have $\sfD^\mathrm{perf}(R)$ of dimension $0$ by Proposition~\ref{prop:Rdim}. We also see that even when $R$ does have finite global dimension it may have nothing to do with the Rouquier dimension of $\sfD^\mathrm{perf}(R)$.



\section{A corrected statement}

Let us give a corrected statement for coherent rings (not necessarily commutative). We fix a right coherent ring $A$ and work with right $A$-modules. The point is that, in this generality, $\sfD^\mathrm{perf}(A)$ can only tell us about the \emph{weak} global dimension, i.e.\ the supremum of the flat dimensions of all $A$-modules. In the noetherian case this coincides with the global dimension and so one recovers the expected result.

\begin{thm}\label{thm:affine}
Let $A$ be a right coherent ring. Then the following are equivalent:
\begin{itemize}
\item[(i)] $A$ has finite weak global dimension;
\item[(ii)] there is a uniform finite bound on the projective dimension of finitely presented $A$-modules;
\item[(iii)] $\sfD^\mathrm{perf}(A)$ has finite Rouquier dimension.
\end{itemize}
\end{thm}
\begin{proof}
We begin with the equivalence of (i) and (ii). Let us suppose that $A$ has finite weak dimension, say $n$, and let $M$ be finitely presented. Since $A$ is coherent, $M$ has a resolution by finitely presented projective $A$-modules. By assumption the flat dimension of $M$ is at most $n$ and so the syzygy theorem tells us that the $n$th syzygy in such a resolution is flat. Thus this syzygy is flat and finitely presented, and hence is projective. This proves that $M$ has projective dimension at most $n$.

On the other hand, suppose there is a uniform bound $n$ on the projective dimension of the finitely presented $A$-modules. Any module $N$ is a filtered colimit of finitely presented $A$-modules $N = \colim_i M_i$ and we know, since each $M_i$ has projective dimension at most $n$, that
\[
\Tor^A_{n+1}(N,-) = \Tor^A_{n+1}(\colim_i M_i,-) =  \colim_i\Tor^A_{n+1}(M_i,-) = 0.
\]
Thus $N$ has flat dimension at most $n$ and $A$ has finite weak dimension as claimed.

Now let us prove the equivalence of (ii) and (iii). Starting with the assumption of (ii) we deduce that $\sfD^\mathrm{perf}(A)$ contains $\modu A$ and so coincides with $\sfD^\mathrm{b}(\modu A)$. It follows from \cite{MR1626856}*{Theorem~8.3} that $\sfD^\mathrm{perf}(A) = \sfD^\mathrm{b}(\modu A)$ has finite Rouquier dimension.

Now let us suppose that (iii) holds. Combining \cite{GreenleesStevenson17}*{Proposition~4.7} with \cite{rouquier08}*{Theorem~4.16} gives that $\sfD^\mathrm{perf}(A) = \sfD^\mathrm{b}(\modu A)$. In particular, every finitely presented $A$-module has finite projective dimension. Moreover, there is a uniform bound $n$ on the projective dimension of the finitely presented modules\textemdash{}this follows from the ghost lemma (cf.\ \cite{MR2258043}*{Lemma~2.4}). This is precisely condition (ii).
\end{proof}

Using an adjusted version of Neeman's original argument we can extend this to the non-affine setting. We refer to \cite{neeman2017strong} for the notation used in the proof.

\begin{cor}
Let $X$ be a quasi-compact, separated, and coherent scheme. Then the following are equivalent:
\begin{itemize}
\item[(i)] $X$ has finite weak global dimension, i.e.\ there is a uniform bound on the lengths of flat resolutions of quasi-coherent sheaves on $X$;
\item[(ii)] $X$ can be covered by open affine subschemes $\Spec(R_i)$ with each $R_i$ of finite weak global dimension.
\item[(iii)] $\sfD^\mathrm{perf}(X)$ has finite Rouquier dimension.
\end{itemize}
\end{cor}
\begin{proof}
The equivalence of (i) and (ii) is straightforward since $X$ is quasi-compact and flatness and exactness are Zariski local. The statement (iii) implies (ii) is also straightforward given Theorem~\ref{thm:affine}: suppose that $\sfD^\mathrm{perf}(X)$ has finite Rouquier dimension and let $\Spec(R)$ be an open affine subscheme of $X$. Then, up to idempotent completion, $\sfD^\mathrm{perf}(R)$ is a Verdier quotient of $\sfD^\mathrm{perf}(X)$ and so the Rouquier dimension of the former is bounded above by the latter. In particular, $\sfD^\mathrm{perf}(R)$ has finite Rouquier dimension, and it is coherent by assumption, so by Theorem~\ref{thm:affine} the ring $R$ has finite weak global dimension.

Now let us suppose that (ii) holds and prove (iii). We start with the assumption that $X$ has a finite cover by open affine subschemes $\Spec(R_i)$ with each $R_i$ coherent and of finite weak global dimension, and hence with $\sfD^\mathrm{perf}(R_i)$ of finite Rouquier dimension by Theorem~\ref{thm:affine}. We proceed, as in the proof of \cite{neeman2017strong}*{Theorem~2.1}, by induction on the number of open affines in the cover. Theorem~\ref{thm:affine} serves as the base case.

Suppose then that we know the result for covers by $n$ open affines and let $R_i$ with $i=1,\ldots,n+1$ give a cover of $X$ by coherent open affine subchemes $U_i$ of finite weak global dimension. We set $U = U_{n+1}$, $V = \bigcup_{i=1}^n U_i$ and $W = U\cap V$, with inclusions $j_U$, $j_V$ and $j_{W}$, and let $G_U, G_V, G_W$ and $G$ be generators for the categories of perfect complexes on $U,V,W$ and $X$. The subscheme $W$ has a cover by the open affines $U_i\cap U$ for $i=1,\ldots,n$ and so the inductive hypothesis applies to it as well as to $V$. By \cite{neeman2017strong}*{Theorem~6.2} there exists an $N$ such that each of $\mathbf{R}j_{U*}G_U$, $\mathbf{R}j_{V*}G_V$, and $\mathbf{R}j_{W*}G_W$ lie in $\mathrm{Coprod}_N(G(-\infty,\infty))$. By the induction hypothesis $\sfD^\mathrm{perf}(U)$, $\sfD^\mathrm{perf}(V)$, and $\sfD^\mathrm{perf}(W)$ have finite Rouquier dimension and so there is an $M$ such that $\sfD^\mathrm{perf}(Y) \subseteq \mathrm{Coprod}_M(G_Y(-\infty,\infty))$ for $Y\in \{U,V,W\}$. Combining these statements we deduce that each of $\mathbf{R}j_{U*}\sfD^\mathrm{perf}(U)$, $\mathbf{R}j_{V*}\sfD^\mathrm{perf}(V)$ and $\mathbf{R}j_{W*}\sfD^\mathrm{perf}(W)$ are contained in $\mathrm{Coprod}_{MN}(G(-\infty,\infty))$.

Using the Mayer-Vietoris triangle for the cover $X = U\cup V$ we see that 
\[
\sfD^\mathrm{perf}(X) \subseteq \mathbf{R}j_{W*}\sfD^\mathrm{perf}(W)\star(\mathbf{R}j_{U*}\sfD^\mathrm{perf}(U) \oplus \mathbf{R}j_{V*}\sfD^\mathrm{perf}(V)) \subseteq \mathrm{Coprod}_{2MN}(G(-\infty,\infty)).
\]
It then follows from \cite{neeman2017strong}*{Proposition~1.9(i)} that $\sfD^\mathrm{perf}(X)$ is generated in at most $2MN$ steps by $G$ and so has finite Rouquier dimension.
\end{proof}

\begin{ack}
Thanks are due to Isambard Goodbody, Srikanth Iyengar, Henning Krause, Amnon Neeman, Kabeer Manali Rahul, and Jeremy Rickard.
\end{ack}




\begin{thebibliography}{10}

\bibitem{MR1626856}Christensen, J. Ideals in triangulated categories: phantoms, ghosts and skeleta. {\em Adv. Math.}. \textbf{136}, 284-339 (1998)
\bibitem{GreenleesStevenson17}Greenlees, J. \& Stevenson, G. Morita theory and singularity categories. {\em Adv. Math.}. \textbf{365} pp. 107055, 51 (2020)
\bibitem{MR2258043}Krause, H. \& Kussin, D. Rouquier's theorem on representation dimension. {\em Trends In Representation Theory Of Algebras And Related Topics}. \textbf{406} pp. 95-103 (2006)
\bibitem{neeman2017strong}Neeman, A. Strong generators in ${\sfD}^{\mathrm{perf}}({X})$ and ${\sfD}^\mathrm{b}_{\coh}({X})$. {\em Ann. Of Math. (2)}. \textbf{193}, 689-732 (2021)
\bibitem{OsofskyVNR}Osofsky, B. Homological dimension and cardinality. {\em Trans. Amer. Math. Soc.}. \textbf{151} pp. 641-649 (1970)
\bibitem{MR0229695}Pierce, R. The global dimension of Boolean rings. {\em J. Algebra}. \textbf{7} pp. 91-99 (1967)
\bibitem{Prestbook}Prest, M. Purity, spectra and localisation. (Cambridge University Press, Cambridge,2009)
\bibitem{rouquier08}Rouquier, R. Dimensions of triangulated categories. {\em J. K-Theory}. \textbf{1}, 193-256 (2008)

\end{thebibliography}

\bibliographystyle{amsplain}

\end{document}